\documentclass{amsart}

\usepackage{amsmath, amssymb, mathrsfs, amsthm, xifthen, verbatim, color, mathrsfs}

\usepackage[arrow, matrix]{xy}

\definecolor{DarkBlue}{rgb}{0,0.2,0.6}
\definecolor{PinkPurple}{rgb}{0.8,0.3,0.3}

\usepackage[pdftex,
            pdfauthor={Mehdi Ghasemi},
            pdftitle={On the topologies induced by a cone},
            pdfsubject={Moment Problem},
            pdfkeywords={Positivity, Cone, Topology, Moment Problem},
            pdfproducer={Latex with hyperref},
            pdfcreator={PDFLaTeX},
            linkcolor=DarkBlue,
            citecolor=PinkPurple,
            colorlinks=true]{hyperref}

\newtheorem{thm}{Theorem}[section]
\newtheorem{lemma}[thm]{Lemma}
\newtheorem{prop}[thm]{Proposition}
\newtheorem{crl}[thm]{Corollary}

\theoremstyle{definition}
\newtheorem{dfn}[thm]{Definition}

\newtheorem{rem}[thm]{Remark}

\newcommand{\reals}{\mathbb{R}}

\newcommand{\rx}{\sgr{\mathbb{R}}{\ux}}
\newcommand{\sos}{\sum\mathbb{R}[\underline{X}]^2}

\newcommand{\ux}{\underline{X}}

\newcommand{\K}[1]{\mathcal{K}_{#1}}
\newcommand{\pos}[1]{\mbox{Psd}(#1)}
\newcommand{\arch}[1]{\mbox{Arch}(#1)}

\newcommand{\X}[1]{\mathcal{X}(#1)}

\newcommand{\T}[1]{\mathcal{T}_{#1}}

\newcommand{\sgr}[2]{#1[#2]}

\newcommand{\norm}[2]{\|\ifthenelse{\isempty{#2}}{\cdot}{#2}\|_{#1}}
\newcommand{\cl}[2]{\overline{#2}^{\ifthenelse{\isempty{#1}}{}{#1}}}
\newcommand{\Cnt}[2]{\mathrm{C}_{\ifthenelse{\isempty{#1}}{}{#1}}(#2)}
\newcommand{\Psd}[2]{\mbox{Psd}_{\ifthenelse{\isempty{#1}}{}{#1}}(#2)}
\newcommand{\Bnd}[2]{\mbox{B}_{\ifthenelse{\isempty{#1}}{}{#1}}(#2)}
\newcommand{\Sp}[2]{\mathfrak{sp}_{\ifthenelse{\isempty{#1}}{}{#1}}(#2)}
\newcommand{\map}[3]{#1:#2\longrightarrow #3}

\begin{document}

\title[On the topologies induced by a cone]{On the topologies induced by a cone}
\author[M. Ghasemi]{Mehdi Ghasemi}
\address{$^2$School of Physical \& Mathematical Sciences,\newline\indent
Nanyang Technological University,\newline\indent
21 Nanyang Link, 637371, Singapore}
\email{mghasemi@ntu.edu.sg}
\date{\today}
\begin{abstract}
Let $A$ be a commutative and unital $\reals$-algebra, and $M$ be an Archimedean quadratic module of $A$. We define a submultiplicative
seminorm $\norm{M}{}$ on $A$, associated with $M$. We show that the closure of $M$ with respect to $\norm{M}{}$-topology is equal
to the closure of $M$ with respect to the finest locally convex topology on $A$. We also compute the closure of any cone in 
$\norm{M}{}$-topology. Then we omit the Archimedean condition and show that there still exists a lmc topology associated to $M$, 
pursuing the same properties.
\end{abstract}
\maketitle
\section{Introduction}
The classical $K$-moment problem for a closed subset $K$ of $\reals^n$, $n\ge1$, is determining whether a given linear functional 
$\map{L}{\rx=\reals[X_1,\dots,X_n]}{\reals}$ is representable as an integral with respect to a positive Radon measure $\mu$, supported
on $K$ or not. In symbols
\[
	L(f)=\int_Kf~\rm{d}\mu.
\]
An obvious necessary condition for existence of such a measure is that for every $f\in\rx$ with $f\ge0$ on $K$, $L(f)$ 
should be non-negative. In 1936, Haviland proved that this necessary condition is also sufficient \cite{Hav01, Hav02}:
\begin{thm}[Haviland]
A linear function $\map{L}{\rx}{\reals}$ is representable as an integral with respect to a positive Radon measure
$\mu$ on $K$ if and only if $L(\pos{K})\subseteq\reals_{\geq0}$.
\end{thm}
Here, $\pos{K}=\{f\in\rx:f(x)\ge0\quad\forall x\in K\}$ and $\reals_{\ge0}=[0,\infty)$. 

The major flaw of Haviland's result is that the structure of $\pos{K}$ is usually very complicated and hence checking non-negativity 
of $L$ on $\pos{K}$ is practically infeasible. 

Schm\"{u}dgen assumed $K$ to be a basic compact semialgebraic sets and solved the $K$-moment problem in this particular
case effectively: A set $K\subseteq\reals^n$ is called basic closed semialgebraic, if there exists a finite set of polynomials
$S=\{g_1,\dots,g_m\}$ such that 
\[
	K=\K{S}=\{x\in\reals^n~:~g_i(x)\ge0,~i=1,\dots,m\}.
\]
The preordering generated by $S$, denoted by $T_S$, is the subset of $\rx$, consisting of all polynomials $f\in\rx$, such that 
\[
	f=\sigma_0+\sum_{e\in\{0,1\}^m}\sigma_e\underline{g}^e,
\]
where $\sigma_0, \sigma_e$, for $e=(e_1,\dots,e_m)\in\{0,1\}^m$ are finite sums of squares of polynomials and 
$\underline{g}^e:=g_1^{e_1}\cdots g_m^{e_m}$.
\begin{thm}[Schm\"{u}dgen \cite{Schm01}]\label{Schm}
Let $K=\K{S}$ be a a basic compact semialgebraic subset of $\reals^n$ and $\map{L}{\rx}{\reals}$ a functional. 
If $L(T_S)\subseteq\reals_{\ge0}$, then $L(\pos{K})\subseteq\reals_{\ge0}$.
\end{thm}
Now, since elements of $T_S$ are finitely representable by polynomials in $S$, checking non-negativity of $L$ over $T_S$ is 
practical and if $\K{S}$ is compact, Schm\"{u}dgen's theorem guarantees non-negativity of $L$ on $\pos{\K{S}}$. Therefore,
in this case, one can use Haviland's theorem to deduce existence of a representing Radon measure for $L$.

A closer look at theorem \ref{Schm} reveals an equivalent topological statement. Let $\varphi$ be the finest non-discrete locally
convex topology on $\rx$. Then \ref{Schm} can be read as
\begin{equation}\label{SchmTop}
	\cl{\varphi}{T_S}=\pos{\K{S}}.
\end{equation}
We explain this in Remark \ref{ddual}.

Following the topological approach, Berg and Maserick in \cite{CB-PHM} showed that 
\begin{equation}\label{BergTop}
	\cl{\norm{1}{}}{\sos}=\pos{[-1,1]^n},
\end{equation}
where $\sos$ is the set of all finite sums of squares of polynomials. In terms of moments, if $L$ is positive semidefinite and 
$\norm{1}{}$-continuous, then $L$ admits integral representation by a Radon measure on $[-1,1]^n$, where 
$\norm{1}{\sum_{\alpha}f_{\alpha}\ux^{\alpha}}=\sum_{\alpha}|f_{\alpha}|$. 
They also generalized this for weighted $\norm{1}{}$-topologies.
In both \eqref{SchmTop} and \eqref{BergTop}, the left side of the equality is the closure of a cone and the right side is $\pos{K}$
for some $K\subseteq\reals^n$. Relaxing the relation between objects of these equations, in \cite{GKM} we started to study the 
following general equation:
\begin{equation}\label{GenMntEq}
	\cl{\tau}{C}=\pos{K},
\end{equation}
where $C\subseteq\rx$ is a cone, $\tau$ a locally convex topology on $\rx$ and $K\subseteq\reals^n$ a closed set. It is also explained
that if \eqref{GenMntEq} holds, then every $\tau$-continuous linear functional with $L(C)\subseteq\reals_{\ge0}$ admits an integral
representation with respect to a Radon measure on $K$. Furthermore, we replaced $\rx$ with a unital commutative $\reals$-algebra and
$\reals^n$ with $\X{A}$, the set of all real valued $\reals$-algebra homomorphisms on $A$, equipped with subspace topology, where 
$\X{A}$ is considered as a subspace of $\reals^A$, with product topology.

In section \ref{KTop-LMC}, we briefly review the solutions of \eqref{GenMntEq}, studied in \cite{GK} and \cite{GKM}. First we fix a
closed set $K\subseteq\X{A}$ and solve \eqref{GenMntEq} for given cones $C$, in terms of the topology $\tau$, which slightly generalizes
results of \cite{GK}. Then we fix a locally multiplicatively convex topology $\tau$ on $A$ and for a given cone $C$ we solve 
\eqref{GenMntEq} in terms of $K$.

In section \ref{CnTop}, we associate a topology to any Archimedean cone $C$ and study its properties.
Then in section \ref{ClCnTop}, we fix an Archimedean cone $C$ and for given sets $K$, solve \eqref{GenMntEq} in terms of the topology 
$\tau$, introduced in section \ref{CnTop}. Moreover, we generalize our results to the case where $C$ is not Archimedean.
\section{Solutions of \eqref{GenMntEq} for a fixed $K$ or a fixed topology}
\label{KTop-LMC}
In this section we briefly review known solutions of \eqref{GenMntEq}, studied in \cite{GK} and \cite{GKM}.
We begin by introducing terms and notations that will be used in this article.

From now on, we always assume that $A$ is a unital commutative $\reals$-algebra. A cone of $A$ is a set $C\subseteq A$ such that
\[
	C+C\subseteq C,\quad \reals_{\ge0}\cdot C\subseteq C.
\]
A quadratic module $M$, is a cone, containing $0$ and $1$ which is closed under multiplication by sums of squares; i.e.,
\[
	\sum A^2\cdot M\subseteq M.
\]
A cone $C$ is said to be \textit{Archimedean}, if for every $a\in A$, there exists $r\in\reals_{\ge0}$ such that $r\pm a\in C$.
If a quadratic module $M$, is also closed under multiplication (i.e. $M.M\subseteq M$) then we say that $M$ is a preordering.

Suppose that $M$ is a quadratic module, then it is easy to see that $I=M\cap-M$ is an ideal of $A$: Clearly $0\in I$ and 
$a^2I\subseteq I$ for every $a\in A$. Thus 
\[
	aI=\left((\frac{a+1}{2})^2-(\frac{a-1}{2})^2\right)I\in I-I\subseteq I.
\]
The ideal $I$ is called the \textit{support} of $M$. Clearly, $M$ is a proper subset of $A$ if and only if $-1\not\in M$.

The set of all real valued $\reals$-algebra homomorphisms on $A$ is denoted by $\X{A}$.
If $\tau$  is a locally convex topology on $A$, the set of all $\tau$-continuous elements of $\X{A}$ will be denoted by 
$\Sp{\tau}{A}$ which is known in the literature as the Gelfand spectrum of $(A,\tau)$.
Every element $a$ of $A$ induces a map on $\X{A}$, in the following way:
\[
\begin{array}{rrll}
	\hat{a}: & \X{A} & \longrightarrow & \reals\\
		& \alpha & \mapsto & \alpha(a).
\end{array}
\]
We denote the set of all elements of $\X{A}$ that are non-negative on $M$ by $\K{M}$. In symbols:
\[
	\K{M}:=\{\alpha\in\X{A}~:~\alpha(M)\subseteq\reals_{\ge0}\}.
\]
\begin{dfn}~
\begin{enumerate}
	\item{A set $U\subseteq A$ is multiplicatively closed, if $U\cdot U\subseteq U$.}
	\item{A \textit{locally multiplicatively convex topology} (lmc) on $A$ is a locally convex topology which admits a 
	system of neighbourhoods at $0$, consisting of multiplicatively closed convex sets.}
	\item{A seminorm $\rho$ on $A$ is called \textit{submultiplicative}, if $\rho(ab)\leq\rho(a)\rho(b)$, for all $a,b\in A$.}
\end{enumerate}
\end{dfn}
We recall the following result about lmc topologies.
\begin{thm}\label{LMCTop}
A locally convex vector space topology $\tau$ on $A$ is lmc if and only if $\tau$ is generated by a family of submultiplicative 
seminorms on $A$.
\end{thm}
\begin{proof}
See \cite[\S 4.3-2]{BNS}.
\end{proof}
The following result from \cite{GKM} will be used in what follows:
\begin{thm}\label{GKM-main}
Let $\rho$ be a submultiplicative seminorm on $A$ and $M$ be a quadratic module of $A$. Then $\cl{\rho}{M}=\pos{\K{M}\cap\Sp{\rho}{A}}$.
\end{thm}
\begin{proof}
See \cite[Theorem 3.7]{GKM}.
\end{proof}
\subsection{The case where $K$ is fixed}
Now lets fix a set $\emptyset\neq X\subseteq\X{A}$. We introduce a topology $\T{X}$ on $A$ such that \eqref{GenMntEq} holds for $C$ to be 
any quadratic module $M$, $\tau=\T{X}$ and $K=X\cap\K{M}$; i.e.,
\[
	\cl{\T{X}}{M}=\pos{X\cap\K{M}}.
\]
Since $X\neq\emptyset$, the family $k(X)=\{D\subseteq X : \emptyset\neq D\textrm{ is compact}\}$ is non-empty. To every $D\in k(X)$ we 
assign a submultiplicative seminorm $\rho_D$ defined by 
\[
	\rho_D(a)=\sup_{\alpha\in D}|\hat{a}(\alpha)|.
\]
\begin{dfn}
The family $\{\rho_D\}_{D\in k(X)}$ induces a locally multiplicatively convex topology $\T{X}$, on $A$.
\end{dfn}
We note that the topology $\T{X}$ defined above is slightly different from the one defined in \cite{GK}. In the former, the topology
$\T{X}$ is induced by the family of seminorms, induced by evaluations on each single point of $X$, i.e., $k(X)=\{\{\alpha\}:\alpha\in X\}$.
\begin{lemma}
Let $D\subseteq\X{A}$ be a compact set. Then $\Sp{\rho_D}{A}=D$.
\end{lemma}
\begin{proof}
Note that an element $\alpha\in\X{A}$ is $\rho_D$-continuous if and only if $|\alpha(a)|\leq\rho_D(a)$ for all $a\in A$.
Thus every $\alpha\in D$ is $\rho_D$-continuous. Let $\beta\in\X{A}\setminus D$ and $D'=D\cup\{\beta\}$. Since $\X{A}$ is completely
regular, there exists a continuous function $f\in\Cnt{}{\X{A}}$ such that $f(\beta)=1$ and $f|_D=0$. 
Note that $\rho_D$ is also extendible to $\Cnt{}{\X{A}}$ naturally by defining $\rho_D(g)=\sup_{\alpha\in D}|g(D)|$ for each 
$g\in\Cnt{}{\X{A}}$. By Stone--Weierstrass theorem, for every $\epsilon>0$, there exist $a_{\epsilon}\in A$ such that 
$\rho_{D'}(f-a_{\epsilon})<\epsilon$. Clearly $|1-\beta(a_{\epsilon})|\leq\epsilon$ and $\rho_D(a_{\epsilon})\leq\epsilon$. So for
$\epsilon<\frac{1}{2}$, we get $\rho_D(a_{\epsilon})<|\beta(a_{\epsilon})|$ which violates the necessary and sufficient condition for
continuity of $\beta$. Thus $\beta\not\in\Sp{\rho_D}{A}$.
\end{proof}
\begin{crl}
$\Sp{\T{X}}{A}=X$.
\end{crl}
\begin{proof}
\[
\begin{array}{lcl}
	\Sp{\T{X}}{A} & = & \bigcup_{D\in k(X)}\Sp{\rho_D}{A}\\
		& = & \bigcup_{D\in k(X)}D\\
		& = & X.
\end{array}
\]
\end{proof}
\begin{thm}\label{GK-main}
Let $M$ be a quadratic module of $A$ and $X\subseteq\X{A}$. Then $\cl{\T{X}}{M}=\pos{X\cap\K{M}}$.
\end{thm}
\begin{proof}
Applying theorem \ref{GKM-main}, for every $D\in k(X)$, $\cl{\rho_D}{M}=\pos{D\cap\K{M}}$. Therefore
\[
\begin{array}{lcl}
	\cl{\T{X}}{M} & = & \bigcap_{D\in k(X)}\cl{\rho_D}{M}\\
		& = & \bigcap_{D\in k(X)}\pos{D\cap\K{M}}\\
		& = & \pos{\bigcup_{D\in k(X)}D\cap\K{M}}\\
		& = & \pos{X\cap\K{M}},
\end{array}
\]
as desired.
\end{proof}
\subsection{The case where $\tau$ is fixed}
We now review the situation where a lmc topology $\tau$ on $A$ is fixed, a quadratic module $M$ is given, and solve \eqref{GenMntEq}
for $K$ as it is explained in \cite[\S 5]{GKM}.

Suppose that $\tau$ is a lmc topology. By theorem \ref{LMCTop}, there exists a family $\mathcal{F}$ of submultiplicative seminorms,
inducing $\tau$ on $A$. For $\rho_1,\rho_2\in\mathcal{F}$, the map defined by $\rho(a)=\max\{\rho_1(a),\rho_2(a)\}$ is again a 
submultiplicative seminorms and the topology induced by $\mathcal{F}\cup\{\rho\}$ is again equal to $\tau$. Inductively, if we add 
the maximum of any finite number of elements of $\mathcal{F}$ to it, the resulting topology will not change. A family of seminorms 
which contains the maximum of all finite sets of its elements is called saturated. Clearly every family of seminorms can be completed 
to a saturated one. The advantage of working with saturated families over non saturated is explained in the next proposition.
\begin{prop}
Suppose $\tau$ is an lmc topology on $A$ generated by a saturated family $\mathcal{F}$ of submultiplicative seminorms of $A$. Then
$\Sp{\tau}{A}=\bigcup_{\rho\in\mathcal{F}}\Sp{\rho}{A}$.
\end{prop}
\begin{proof}
\cite[\S 4.10-7]{BNS}.
\end{proof}
Implementing the same argument we used in theorem \ref{GK-main} for a saturated family of submultiplicative seminorms $\mathcal{F}$
inducing $\tau$ and a quadratic module $M$, we have,
\[
\begin{array}{lcl}
	\cl{\tau}{M} & = & \bigcap_{\rho\in\mathcal{F}}\cl{\rho}{M}\\
		& = & \bigcap_{\rho\in\mathcal{F}}\pos{\Sp{\rho}{A}\cap\K{M}}\\
		& = & \pos{\bigcup_{\rho\in\mathcal{F}}\Sp{\rho}{A}\cap\K{M}}\\
		& = & \pos{\Sp{\tau}{A}\cap\K{M}},
\end{array}
\]
which proves the following:
\begin{thm}
Let $\tau$ be an lmc topology on $A$ and let $M$ be any quadratic module of $A$. Then $\cl{\tau}{M}=\pos{\Sp{\tau}{A}\cap\K{M}}$.
\end{thm}
\section{The seminorm induced by a cone}
\label{CnTop}
Fixing an Archimedean quadratic module $M$ of $A$, we associate a non-negative function $\norm{M}{}$ to $M$, defined on $A$. We
prove that it is in fact a submultiplicative seminorm on $A$ and study some of its basic properties, such as its relation to the
finest locally convex topology on $A$ in the following sections.
\begin{dfn}
Let $M$ be an Archimedean quadratic module of $A$. For every $a\in A$, define $\norm{M}{a}$ by
\[
	\norm{M}{a}:=\inf\{r\in\reals~:~r\pm a\in M\}.
\]
\end{dfn}
We will make use of the following well-known result of T. Jacobi to prove some basic properties of $\norm{M}{}$:
\begin{thm}[Jacobi]\label{Jacobi} Suppose $M$ is an Archimedean quadratic module of $A$. Then for each $a\in A$,
\[
	\hat{a}>0 ~ on ~ \K{M}\Rightarrow a\in M.
\]
\end{thm}
\begin{proof}
See \cite[Theorem 4]{J}.
\end{proof}
\begin{prop}\label{non-neg}
For all $a\in A$, $\norm{M}{a}\ge0$.
\end{prop}
\begin{proof}
There are two possible cases:

\textit{Case 1}: $a\in M$ or $-a\in M$. Without loss of generality, assume that $a\in M$. Then for all $\alpha\in\K{M}$,
$\hat{a}(\alpha)\ge0$. Therefore, for any real number $r>0$, $\widehat{r+a}>0$ on $\K{M}$ and hence by theorem \ref{Jacobi}, 
$r+a\in M$. 

If also $-a\in M$, then $\alpha(a)=0$ for all $\alpha\in\K{M}$. Thus if $r<0$, then $r\pm a\not\in M$ and 
$\norm{M}{a}=0$. 

Suppose that $-a\not\in M$. Then, there exists $\alpha\in\K{M}$ such that $\alpha(-a)\leq0$. Therefore 
$r-a\not\in M$ for any $r\leq0$ and hence $\norm{M}{a}>0$.

\textit{Case 2}: $\pm a\not\in M$. By theorem \ref{Jacobi}, there exist $\alpha,\beta\in\K{M}$ such that $\alpha(a)\leq0$ and
$\beta(-a)\leq0$. Therefore, $\alpha(r+a),\beta(r-a)\leq0$ for every $r\leq0$ which implies that $r\pm a\not\in M$. So the set
$\{r\in\reals:r\pm a\in M\}$ is bounded below by $0$. Hence $\norm{M}{a}$ exists and is non-negative.
\end{proof}
In fact $\norm{M}{}$ induces a submultiplicative seminorm on $A$ (A seminorm $\rho$ on $A$ is said to be submultiplicative if 
$\rho(a\cdot b)\leq\rho(a)\rho(b)$ for all $a,b\in A$).
\begin{prop}\label{smsn}
$\norm{M}{}$ is a submultiplicative seminorm on $A$.
\end{prop}
\begin{proof}
By proposition \ref{non-neg}, $\norm{M}{0}=0$ and the range of $\norm{M}{}$ consists of non-negative real numbers.
\begin{itemize}
\item[(1)]{
$\forall\lambda\in\reals\ \forall a\in A\quad\norm{M}{\lambda a}=|\lambda|\norm{M}{a}$:

Clearly for $\lambda=0$, $\norm{M}{\lambda a}=\norm{M}{0}=0\times\norm{M}{a}=0$. Suppose that $\lambda\neq0$. Then
\[
\begin{array}{lcl}
	\{r:r\pm\lambda a\in M\} & = & \{r:|\lambda|(\frac{r}{|\lambda|}\pm a)\in M\}\\
	 & = & \{r:\frac{r}{|\lambda|}\pm a\in\frac{1}{|\lambda|}M\}\\
	 & = & |\lambda|\{r:r\pm a\in M\}.
\end{array}
\]
So, $\norm{M}{\lambda a}=|\lambda|\norm{M}{a}$.
}
\item[(2)]{
$\forall a,b\in A\quad\norm{M}{a+b}\leq\norm{M}{a}+\norm{M}{b}$:

For any $\epsilon>0$ and $a,b\in A$, we have
\[
	\norm{M}{a}+\norm{M}{b}+\epsilon\in\{r:r\pm(a+b)\in M\}.
\]
Hence, $\sup\{r:r\pm(a+b)\in M\}\leq\norm{M}{a}+\norm{M}{b}+\epsilon$, for every $\epsilon>0$.
So, $\norm{M}{a+b}\leq\norm{M}{a}+\norm{M}{b}$.
}
\item[(3)]{
$\forall a,b\in A\quad\norm{M}{ab}\leq\norm{M}{a}\norm{M}{b}$:

For any $\epsilon>0$, $(\norm{M}{a}+\epsilon)\pm a,(\norm{M}{b}+\epsilon)\pm b\in M$. Therefore $|\alpha(a)|<\norm{M}{a}+\epsilon$ 
and $|\alpha(b)|<\norm{M}{b}+\epsilon$ for all $\alpha\in\K{M}$. So 
\[
	|\alpha(ab)|<(\norm{M}{a}+\epsilon)(\norm{M}{b}+\epsilon),
\]
and hence $(\norm{M}{a}+\epsilon)(\norm{M}{b}+\epsilon)\pm ab\in M$ by \ref{Jacobi}. Therefore 
$\norm{M}{ab}\leq(\norm{M}{a}+\epsilon)(\norm{M}{b}+\epsilon)$ for any $\epsilon>0$. Letting $\epsilon\longrightarrow0$, we get
\[
	\norm{M}{ab}\leq\norm{M}{a}\norm{M}{b},
\]
as desired.
}
\end{itemize}
\end{proof}
Proposition \ref{smsn}, simply asserts that if $M\subseteq A$ is an Archimedean quadratic module then the pair $(A,\norm{M}{})$ 
is a seminormed $\reals$-algebra.
\section{Closure with respect to the induced seminorm}
\label{ClCnTop}
In this section we study the relation between $\norm{M}{}$-continuity and positivity of linear functionals on $M$, when $M$ is an
Archimedean quadratic module. Then we find the $\norm{M}{}$-closure of $M$. This gives a  solution for integral representability of 
positive semidefinite $\norm{M}{}$-continuous functionals.
\begin{thm}\label{psd-cnt}
Let $M$ be an Archimedean quadratic module. If a linear functional $\map{L}{A}{\reals}$ is non-negative on $M$, then $L$ is
$\norm{M}{}$-continuous.
\end{thm}
\begin{proof}
Since $M$ is Archimedean, for every $a\in A$, there exists $r\ge0$ such that $r\pm a\in M$. By positivity of $L$ on $M$, we have 
$L(r\pm a)\ge0$. So, $|L(a)|\leq L(r)=rL(1)$ which by definition means $\norm{M}{a}\leq r$. Therefore, for every $a\in A$ we have
$|L(a)|\leq L(1)\cdot\norm{M}{a}$ and so, $L$ is $\norm{M}{}$-continuous.
\end{proof}
\begin{rem}\label{ddual}
Let $\tau$ be a locally convex topology on $A$ and $C\subseteq A$ be a cone in $A$. Let
\[
	C_{\tau}^{\vee}:=\{\map{L}{A}{\reals}~:~L\textrm{ is }\tau\textrm{-continuous and }L(C)\subseteq\reals_{\ge0}\},
\]
and
\[
	C_{\tau}^{\vee\vee}:=\{a\in A~:~L(a)\ge0\quad\forall L\in C_{\tau}^{\vee}\}.
\]
One can show that $C_{\tau}^{\vee\vee}=\cl{\tau}{C}$:

If $b\not\in\cl{\tau}{C}$, then there exists a convex open set $O$ containing $b$ and disjoint from $C$. By Banach separation 
theorem, there exists $L\in C_{\tau}^{\vee}$ such that $L<0$ on $O$. Hence $b\not\in C_{\tau}^{\vee\vee}$ which proves 
$C_{\tau}^{\vee\vee}\subseteq\cl{\tau}{C}$. The reverse inclusion is clear. Note that $C\subseteq C_{\tau}^{\vee\vee}$, 
therefore $\cl{\tau}{C}\subseteq \cl{\tau}{C_{\tau}^{\vee\vee}}$ and
\[
	C_{\tau}^{\vee\vee}=\bigcap_{L\in C_{\tau}^{\vee}}L^{-1}(\reals_{\ge0}),
\]
is $\tau$-closed; i.e., $\cl{\tau}{C_{\tau}^{\vee\vee}}=C_{\tau}^{\vee\vee}$.
\end{rem}
\begin{crl}\label{sp-normM}
Let $M$ be an Archimedean quadratic module and $\varphi$ be the finest locally convex topology on $A$. 
Then $\cl{\varphi}{M}=\cl{\norm{M}{}}{M}$.
\end{crl}
\begin{proof}
Since $\varphi$ is the finest locally convex topology on $A$, every linear functional on $A$ is $\varphi$-continuous. 
By definition $M_{\varphi}^{\vee}=\{\map{L}{A}{\reals}:~L(M)\subseteq\reals_{\ge0}\}$. Since $M$ is Archimedean, by theorem
\ref{psd-cnt}, $L(M)\subseteq\reals_{\ge0}$ implies that $L$ is $\norm{M}{}$-continuous, so $M_{\varphi}^{\vee}=M_{\norm{M}{}}^{\vee}$. 
Therefore, $M_{\varphi}^{\vee\vee}=M_{\norm{M}{}}^{\vee\vee}$ and by applying Remark \ref{ddual}, we get $\cl{\varphi}{M}=\cl{\norm{M}{}}{M}$.
\end{proof}
\begin{thm}\label{sn-cls}
Let $M$ be an Archimedean quadratic module of $A$ and $T$ a cone such that $\K{M}\cap\K{T}\neq\emptyset$. Then
\[
	\cl{\norm{M}{}}{T}=\Psd{}{\K{M}\cap\K{T}}.
\]
\end{thm}
\begin{proof}
($\subseteq$)
Clearly $T\subseteq\Psd{}{\K{T}}\subseteq\Psd{}{\K{M}\cap\K{T}}$. Since every $\alpha\in\K{T}\cap\K{M}$ is $\norm{M}{}$-continuous 
and 
\[
	\Psd{}{\K{M}\cap\K{T}}=\bigcap_{\alpha\in\K{M}\cap\K{T}}\alpha^{-1}(\reals_{\ge0}),
\]
we see that $\Psd{}{\K{M}\cap\K{T}}$ is $\norm{M}{}$-closed. Therefore $\cl{\norm{M}{}}{T}\subseteq\Psd{}{\K{M}\cap\K{T}}$.

($\supseteq$)
Let $A_1^{\ast}=\{L\in A^{\ast}:L(1)=1\}$, where $A^{\ast}$ is the dual of $(A,\norm{M}{})$ equipped with weak-$\ast$ topology. 
The set $T_{\norm{M}{},1}^{\vee}:=A_1^{\ast}\cap T_{\norm{M}{}}^{\vee}$ is a convex closed subset of a bounded closed ball of 
$A^{\ast}$, which is compact by Banach-Alaoglu theorem. By Krein-Milman theorem, $T_{\norm{M}{},1}^{\vee}$ is the weak-$\ast$ 
closure of convex hull of its extreme points.

\textit{\textbf{Claim.}} If $L$ is an extreme point of $T_{\norm{M}{},1}^{\vee}$, then $L\in\K{T}\cap\K{M}$.

\textit{Proof of the Claim:} For any $c\in A$ such that $L(c)\neq0$ the map defined by $L_c(x)=\frac{L(cx)}{L(c)}$ is a
$\norm{M}{}$-continuous linear functional and $L_c(1)=1$. If $L\in T_{\norm{M}{},1}^{\vee}$ and $L(a),L(b)>0$, then
\begin{equation}\label{cvx-cmb}
	L(a)\cdot L_a(x)+L(b)\cdot L_b(x)=L(a+b)\cdot L_{a+b}(x).
\end{equation}
Therefore $L_{a+b}=(\frac{L(a)}{L(a+b)})L_a+(\frac{L(b)}{L(a+b)})L_b$ is a convex combination of $L_a$ and $L_b$. Since $L$ is
continuous, there exists $C>0$ such that for all $x\in A$, $|L(x)|\leq C\norm{M}{x}$. Therefore, for some $N>0$, $L(N+x), L(N-x)>0$.
Rewriting \eqref{cvx-cmb} for $N\pm x$, yields:
\[
	L=L_{2N}=\lambda L_{N+x}+(1-\lambda)L_{N-x}.
\]
By assumption, $L$ is an extreme point, thus $L=L_{N+x}$. So for every $y\in A$,
\[
	L(y)=L_{N+x}(y)=\frac{L(N\cdot y)+L(xy)}{L(N+x)}=\frac{N\cdot L(y)+L(xy)}{N+L(x)}.
\]
Hence $N\cdot L(y)+L(x)L(y)=N\cdot L(y)+L(xy)$ or $L(xy)=L(x)L(y)$, as claimed.

Denoting by $\mbox{cov}(\K{T}\cap\K{M})$, the convex hull of $\K{T}\cap\K{M}$ in $T_{\norm{M}{},1}^{\vee}$, the claim simply 
states that the weak-$\ast$ closure of $\mbox{cov}(\K{T}\cap\K{M})$ is equal to $T_{\norm{M}{},1}^{\vee}$.

Now to show the reverse inclusion, take $a\not\in\cl{\norm{M}{}}{T}$. Then, there exists $L\in T_{\norm{M}{}}^{\vee}$ such $L(a)<0$. 
Replacing $L$ with $\frac{1}{L(1)}L$ if necessary, we can assume that $L\in T_{\norm{M}{},1}^{\vee}$.
Note that every element $a\in A$ defines a continuous functional $\map{\hat{a}}{A^{\ast}}{\reals}$ which $L\mapsto L(a)$.
If $b\in\Psd{}{\K{T}\cap\K{M}}$, then $\hat{b}\ge0$ on $\K{T}\cap\K{M}$. By continuity, $\hat{b}\ge0$ on 
$\cl{\ast}{\mbox{cov}(\K{T}\cap\K{M})}=T_{\norm{M}{},1}^{\vee}$. Therefore
\[
	a\not\in\cl{\norm{M}{}}{T}\Rightarrow a\not\in\Psd{}{\K{T}\cap\K{M}},
\]
or equivalently, $\Psd{}{\K{T}\cap\K{M}}\subseteq\cl{\norm{M}{}}{T}$. This completes the proof.
\end{proof}
\begin{rem}
In the theorem \ref{sn-cls}, if we assume that $T$ is a quadratic module which is a richer structure, then the proof could be 
simplified by directly applying theorem \ref{GKM-main}. We only need to note that $\Sp{\norm{M}{}}{A}=\K{M}$ which is proved 
implicitly in corollary \ref{sp-normM}.
\end{rem}
\section{Non-Archimedean cones}
In this section we consider the case where the quadratic module $M$ is not Archimedean, but $\K{M}\neq\emptyset$.
The purpose of this section is to define a topology $\T{M}$ on $A$ such that \eqref{GenMntEq} holds.
\begin{lemma}\label{ds-lemma}
Let $M_1,M_2\subseteq A$ be Archimedean quadratic modules. Then
\begin{enumerate}
	\item{
	$M_1\cap M_2$ is Archimedean;
	}
	\item{
	If $M_1\subseteq M_2$ then the identity map $\map{\iota}{(A,\norm{M_1}{})}{(A,\norm{M_2}{})}$ is continuous.
	}
\end{enumerate}
\end{lemma}
\begin{proof}
(1) Since $M_1$ and $M_2$ are Archimedean, for every $a\in A$ there exist $r_1,r_2\in\reals$ such that $r_1\pm a\in M_1$ and
$r_2\pm a\in M_2$. Take $r=\max(r_1, r_2)$, we have $r\pm a\in M_1\cap M_2$. 

(2) If $M_1\subseteq M_2$ then for all $a\in A$ we have $\norm{M_2}{a}\leq\norm{M_1}{a}$. 
Thus the identity map $\map{\iota}{(A,\norm{M_1}{})}{(A,\norm{M_2}{})}$ is continuous.
\end{proof}
Suppose that $\K{M}\neq\emptyset$, then there always exists an Archimedean quadratic module $M'$ containing $M$; 
take $\alpha\in\K{M}$ and let 
\[
	M':=\{a\in A~:~\hat{a}(\alpha)\ge0\}.
\]
Clearly $M'$ is Archimedean (proof: $|\hat{a}(\alpha)|\pm a\ge0$ on $\{\alpha\}=\K{M'}$) and $M\subseteq M'$.
Let 
\[
	\arch{M}:=\{T~:~T\textrm{ is Archimedean quadratic module and }M\subseteq T\},
\]
then we can prove the following:
\begin{lemma}
The family $\{(A,\norm{T}{}):T\in\arch{M}\}$ together with identity maps forms a direct system of seminormed algebras.
\end{lemma}
\begin{proof}
The set $\arch{M}$ is partially ordered by inclusion. Take $T_1,T_2\in\arch{M}$, the quadratic module $T_3=T_1\cap T_2$ contained 
in both $T_1$ and $T_2$ and also belongs to $\arch{M}$. The inclusion maps $\map{\iota_1}{(A,\norm{T_3}{})}{(A,\norm{T_1}{})}$
and $\map{\iota_2}{(A,\norm{T_3}{})}{(A,\norm{T_2}{})}$ are continuous by lemma \ref{ds-lemma}. 
So $\{(A,\norm{T}{}):T\in\arch{M}\}$ together with inclusion maps is a direct system.
\end{proof}
The weakest topology on $A$ such that all maps $\map{\iota}{A}{(A,\norm{T}{})}$, $T\in\arch{M}$ are continuous, coincides with the 
direct limit topology of $\{(A,\norm{T}{}):T\in\arch{M}\}$ on $A$. We denote this topology with $\T{M}$. In symbols:
\[
	(A,\T{M})=\varinjlim\limits_{T\in\arch{M}}(A,\norm{T}{}).
\]
\begin{thm}
Let $M$ be a quadratic module and $C$ a cone in $A$ such that $\K{M}\cap\K{C}\neq\emptyset$. Then $\cl{\T{M}}{C}=\Psd{}{\K{M}\cap\K{C}}$.
\end{thm}
\begin{proof}
Since $\T{M}=\varinjlim\limits_{T\in\arch{M}}(A,\norm{T}{})$, applying theorem \ref{sn-cls} we have
\[
\begin{array}{lcl}
	\cl{\T{M}}{C} & = & \bigcap_{T\in\arch{M}}\cl{\norm{T}{}}{C}\\
		& = & \bigcap_{T\in\arch{M}}\Psd{}{\K{T}\cap\K{C}}\\
		& = & \Psd{}{\bigcup_{T\in\arch{M}}\K{T}\cap\K{C}}\\
		& = & \Psd{}{\K{M}\cap\K{C}}.
\end{array}
\]
\end{proof}
%

%
\end{document}